\documentclass[reqno,centertags,12pt,draft]{amsart}
\usepackage[letterpaper,margin=1.3in]{geometry}
\usepackage{amsmath,amsthm,amsfonts,amssymb,enumerate}
\usepackage[bookmarksopen=true,final]{hyperref}
\usepackage[nobysame,abbrev]{amsrefs}

\newcommand{\f}{\frac}

\newcommand{\no}{\notag}

\newcommand{\bb}{\mathbb}
\newcommand{\cc}{\mathcal}
\newcommand{\ol}{\overline}

\newcommand{\ti}{\tilde}

\newcommand{\bs}{\backslash}
\newcommand{\al}{\alpha}

\newcommand{\ga}{\gamma}

\newcommand{\Om}{\Omega}

\newcommand{\te}{\theta}

\newcommand{\eps}{\varepsilon}
\newcommand{\pd}{\partial}
\newcommand{\supp}{\mathrm{supp}}

\newcommand{\ca}{\mathrm{Cap}}
\newcommand{\x}{{\mathbf x}}
\newcommand{\cf}{{\mathbf1}}
\newcommand{\UT}{{\cc U^{\cc T}}}
\newcommand{\UTR}{{\cc U^{\cc T,\cc R}}}

\renewcommand{\Re}{\mathrm{Re}}

\newtheorem{theorem}{Theorem}
\newtheorem{lemma}[theorem]{Lemma}
\newtheorem{corollary}[theorem]{Corollary}
\theoremstyle{definition}

\theoremstyle{remark}

\numberwithin{equation}{section}
\numberwithin{theorem}{section}

\begin{document}
\title{On the Widom factors for $L_p$ extremal polynomials}

\author{G\"{o}kalp Alpan}
\address{Department of Mathematics, Rice University, Houston, TX 77005, USA}
\email{alpan@rice.edu}

\author{Maxim Zinchenko}
\address{Department of Mathematics and Statistics, University of New Mexico, 311 Terrace Street NE, MSC01 1115, Albuquerque, NM 87106, USA}
\email{maxim@math.unm.edu}
\thanks{\footnotesize M.Z. is supported in part by Simons Foundation grant CGM-581256.}


\begin{abstract}
We continue our study of the Widom factors for $L_p(\mu)$ extremal polynomials initiated in \cite{AlpZin20}. In this work we characterize sets for which the lower bounds obtained in \cite{AlpZin20} are saturated, establish continuity of the Widom factors with respect to the measure $\mu$, and show that despite the lower bound $[W_{2,n}(\mu_K)]^2\ge 2S(\mu_K)$ for the equilibrium measure $\mu_K$ on a compact set $K\subset\bb R$ the general lower bound $[W_{p,n}(\mu)]^p\ge S(\mu)$ is optimal even for measures $d\mu=wd\mu_K$ with polynomial weights $w$ on $K\subset\bb R$. We also study pull-back measures under polynomial pre-images introduced in \cite{GerVAs88,PehSte00} and obtain invariance of the Widom factors for such measures. Lastly, we study in detail the Widom factors for orthogonal polynomials with respect to the equilibrium measure on a circular arc and, in particular, find their limit, infimum, and supremum and show that they are strictly monotone increasing with the degree and strictly monotone decreasing with the length of the arc.
\end{abstract}

\date{\today}
\maketitle

\section{Introduction}
Let $K$ be a compact subset of $\bb C$ which contains infinitely many points and $\mu$ be a finite (positive) Borel measure such that $\supp(\mu)=K$. For $0<p\le\infty$ and $n\in \bb N$, a monic polynomial $T_n$ is called an $L_p(\mu)$ extremal polynomial if
\begin{align}
\|T_n\|_{L_p(\mu)}= \inf_{P_n\in \Pi_n} \|P_n\|_{L_p(\mu)}
\end{align}
where $\Pi_n$ is the set of all monic polynomials of degree $n$, $\|P_n\|_{L_\infty (\mu)}=\|P_n\|_{K}:=\mathrm{sup}_{z\in K} |P_n(z)|$ and $\|P_n\|_{L_p(\mu)}=\left(\int|P_n(z)|^pd\mu(z)\right)^{1/p}$, for $0<p<\infty$. We set
\begin{align}
t_{p,n}(\mu) := \inf_{P_n\in \Pi_n} \|P_n\|_{L_p(\mu)}.
\end{align}
We remark that $n$-th $L_p(\mu)$ extremal polynomial is unique if $p\in (1,\infty]$ and not necessarily unique if $p\in (0,1]$. For $p=2$ the extremal polynomials are orthogonal polynomials and for $p=\infty$ they are called Chebyshev polynomials.

We need several concepts from potential theory to discuss Widom factors and refer the reader to  \cite{Ran95,ST97} for potential theoretic preliminaries. Let $\ca(K)$ denote the logarithmic capacity. If $\ca(K)>0$ then we denote the equilibrium measure of $K$ by $\mu_K$. In this case, $g_K$ denotes the Green function, that is,
\begin{align}\label{gK-muK}
g_K(z) = -\log\ca(K)+\int\log|z-\zeta|\,d\mu_K(\zeta), \quad z\in\bb C.
\end{align}
The outer domain $\Om_K$ is the unbounded component of $\ol{\bb C}\bs K$. Throughout, regularity of a set or a point means regularity with respect to the Dirichlet problem in $\Om_K.$

When $K$ is non-polar, Widom factors are defined by
\begin{align}\label{W-def}
W_{p,n}(\mu):= t_{p,n}(\mu)/\ca(K)^n.
\end{align}
For $p=\infty$, we use the notation $W_{\infty,n}(K)$ because in this case this quantity does not depend on the measure. For a fuller treatment of the results on Widom factors we refer the reader to \cite{Alp17, Alp19, AlpGon15, AlpZin20, And17, AndNaz18, Chr12, CSYZ19, CSZ11, CSZ17, CSZ3, CSZ4, Eic17, GonHat15, Sch08, Tot09, Tot11, Tot14, TotVar15, TotYud15, Wid69}.

Let $d\mu= w\,d\mu_K+d\mu_s$ be the Lebesgue decomposition of $\mu$ with respect to $\mu_K$. We introduce the exponential relative entropy of $\mu$ (relative to $\mu_K$) by
\begin{align} \label{S-def}
S(\mu) = S_K(w) := \exp\left[\int\log w(z)\, d\mu_K(z)\right].
\end{align}

It was recently proved that (first in \cite{Alp19} for $p=2$ then for $0<p<\infty$ in \cite{AlpZin20})
\begin{align}\label{LB-univ}
\big[W_{p,n}(\mu)\big]^p \ge S(\mu), \quad n\in\bb N, \; 0<p<\infty.
\end{align}
Although \eqref{LB-univ} is stated in \cite{AlpZin20} for unit measures, it is clear that $\big[W_{p,n}(\mu)\big]^p/S(\mu)$ remains unchanged after normalizing the measure so \eqref{LB-univ} is valid for finite measures.
It was also shown in \cite{AlpZin20} that for any $0<p<\infty$ and $n\in \bb N$,
\begin{align}\label{LB-sharp}
\inf_{\mu} \big[W_{p,n}(\mu)\big]^p/S(\mu) = 1,
\end{align}
where the infimum is taken over probability measures on ${[-2,2]}$.

Nonetheless, the improved lower bound
\begin{align}\label{LB-impr}
\big[W_{2,n}(\mu)\big]^2 \ge 2 S(\mu), \quad n\in\bb N,
\end{align}
holds in the following cases (see Sections 3-5 in \cite{AlpZin20}):
$\mu$ is the equilibrium measure of a non-polar compact subset of $\bb R$, $\mu$ is in the isospectral torus of a finite gap set, and $\mu$ is given by the Jacobi weight for a certain range of parameters.

A natural question that arises is what features of the weight $w$ are responsible for the doubling of the lower bound for $[W_{2,n}(w\mu_K)]^2$. We prove that when $K$ is a non-polar compact subset of $\bb R$, \eqref{LB-sharp} holds if the infimum is taken over all finite measures of the form $\mu=w\mu_K$ where $w$ is a polynomial weight positive on $K$. In particular, this shows that analyticity of the weight is not one of the features which leads to the improved lower bound \eqref{LB-impr}.

In the case of Chebyshev polynomials (i.e., $p=\infty$) the analogous lower bounds for the Widom factors of a non-polar compact set $K$ are well known. When $K\subset\bb C$ (see \cite[Theorem~5.5.4]{Ran95}),
\begin{align}\label{LB-cheb-C}
W_{\infty,n}(K) \ge 1, \quad n\in\bb N,
\end{align}
and when $K\subset\bb R$ (see \cite{Sch08}),
\begin{align}\label{LB-cheb-R}
W_{\infty,n}(K) \ge 2, \quad n\in\bb N.
\end{align}
In Theorem~1.2 and Theorem~1.1 in \cite{CSZ3}, for a fixed $n$, complete characterizations of the sets $K$ for which equality is attained in \eqref{LB-cheb-C} and \eqref{LB-cheb-R}, respectively, were given.

Note that for the equilibrium measure $\mu_K$ of a non-polar set $K$ we have $S(\mu_K)=1$ by \eqref{S-def} and hence
\begin{align}\label{LB-equil-C}
\big[W_{p,n}(\mu_K)\big]^p \ge 1, \quad n\in\bb N, \; 0<p<\infty,
\end{align}
for non-polar compact $K\subset\bb C$ by \eqref{LB-univ} and
\begin{align}\label{LB-equil-R}
\big[W_{2,n}(\mu_K)\big]^2 \ge 2, \quad n\in\bb N,
\end{align}
for non-polar compact $K\subset\bb R$ by \eqref{LB-impr}.

For a fixed $n$, we describe the sets $K$ for which equality is attained in \eqref{LB-equil-C} and \eqref{LB-equil-R}, respectively, provided that the set is regular. It turns out that $[W_{2,n}(\mu_K)]^2$ and $W_{\infty,n}(K) $ realize their respective theoretical lower bounds simultaneously.

Similarities between the asymptotics and bounds for $[W_{2,n}(\mu_K)]^2$ and $W_{\infty,n}(K) $ go well beyond this. See \cite[Corollary~1.5]{Alp19} for a recent result on boundedness from above. When $K$ is a $C^{2+}$ smooth Jordan curve,
$W_{\infty,n}(K) \rightarrow 1$ (Section~8, \cite{Wid69}). It follows from \eqref{LB-equil-C} and the fact that $W_{2,n}(\mu_K)\leq W_{\infty,n}(K)$ (since the $L_p$-norm with respect to a probability measure is non-decreasing in $p$) we also have $W_{2,n}(\mu_K)\rightarrow 1$ in this case.

When $K$ is a circular arc we evaluate $\lim_{n\rightarrow\infty} [W_{2,n}(\mu_K)]^2$ in \eqref{Arc-lim}. Comparing with the asymptotics of the Chebyshev polynomials on the circular arc $K$ \cite[Corollary~2.6]{Eic17} then shows that  $\lim_{n\to\infty} [W_{2,n}(\mu_K)]^2 = \lim_{n\to\infty} W_{\infty,n}(K)$.

Considering the above results, we raise the question on whether $[W_{2,n}(\mu_K)]^2$ and $W_{\infty,n}(K)$ have the same limit when $K$ is a smooth (say $C^{2+})$ Jordan arc. Note that a limit exists for $[W_{2,n}(\mu_K)]^2$ by \cite[Theorem~12.3.]{Wid69} but it is unclear whether there should be a limit for $W_{\infty,n}(K)$.

Given a finite Borel measure $\mu_0$ with compact support in $\bb C$, using a polynomial transformation it is possible to construct a new measure $\mu$ on the polynomial pre-image of $\supp(\mu_0)$ such that the extremal polynomials for $\mu$ and $\mu_0$ have many properties in common, see \cite{GerVAs88}, \cite{PehSte00}. We show invariance of relative entropy and Widom factors under a polynomial transformation. The novelty in our approach is to consider Lebesgue decomposition with respect to the equilibrium measure. We discuss Widom factors for reflectionless measures (see Section~\ref{Sec:Preimages} for the definition). We also derive $L_p({\mu})$ extremal polynomials for $1\leq p<\infty$ on the pre-image of the polynomial $\mathcal T(z)=z^N$ when $\supp(\mu_0)\subset \partial \bb D$.

The plan of the paper is as follows. In Section~\ref{Sec:Convergence}, we show the convergence $W_{p,n}(\mu_j)\rightarrow W_{p,n}(\mu)$ as $j\rightarrow\infty$ if  $\ca(\supp(\mu_j))\to\ca(\supp(\mu))$ and $\mu_j\rightarrow\mu$ in the weak star sense. Then we prove \eqref{LB-sharp} for polynomial weights. In Section~\ref{Sec:Preimages}, we derive several formulas for the extremal polynomials on polynomial pre-images. In Section~\ref{Sec:Saturation}, we investigate precisely when the equality holds in \eqref{LB-equil-C} and  \eqref{LB-equil-R}. In Section~\ref{Sec:Arc}, we give explicit formulas for $\inf_n [W_{2,n}(\mu_K)]^2$, $\sup_n [W_{2,n}(\mu_K)]^2$ and $\lim_{n\to\infty} [W_{2,n}(\mu_K)]^2$ when $K$ is a circular arc and show that the Widom factors $W_{2,n}(\mu_K)$ are strictly increasing with respect to $n$. In addition, for each $n$ fixed, we establish strict monotonicity of $W_{2,n}(\mu_K)$ and $\al_n(\mu_K)$ ($n$-th Verblunsky coefficient) with respect to the length of circular arc.

\section{Widom factors} \label{Sec:Convergence}

In this section we prove continuity of the Widom factors $W_{p,n}(\mu)$ with respect to the measure $\mu$ and as an application strengthen \eqref{LB-sharp} by showing that the lower bound \eqref{LB-univ} is sharp even for measures $d\mu=wd\mu_K$ on an arbitrary non-polar set $K\subset\bb R$ with analytic, in fact, polynomial weight $w$.

\begin{theorem}\label{Thm:mu-conv}
Let $\mu, \mu_j$, $j\in\bb N$, be finite Borel measures supported on a common bounded subset of $\bb C$ and such that $\mu_j\to\mu$ in the weak star sense. Then for any $p\in(0,\infty)$ and $n\in\bb N$,
\begin{align}\label{t-mu-conv}
\lim_{j\to\infty}t_{p,n}(\mu_j) = t_{p,n}(\mu).
\end{align}
In addition, if $\ca(\supp(\mu_j))\to\ca(\supp(\mu))$, in particular, if $\supp(\mu)=\supp(\mu_j)$ for all $j$, then
\begin{align}\label{W-mu-conv}
\lim_{j\to\infty}W_{p,n}(\mu_j) = W_{p,n}(\mu).
\end{align}
\end{theorem}
\begin{proof}
Let $D\subset\bb C$ be a closed disc that supports all the measures. Then the weak star convergence $\mu_j\to\mu$ means $\int fd\mu_j\to\int fd\mu$ for any continuous function $f$ on $D$. By a compactness argument we also have $\int fd\mu_j\to\int fd\mu$ uniformly in $f$ on any compact set in the space of continuous functions on $D$. Since the subspace of polynomials of degree at most $n$ is finite dimensional, any closed bounded set in this subspace is compact and any continuous image of such a set is also compact. Thus, the set $S=\{|Q|^p:Q\in\Pi_n,\;\|Q\|_D\le M\}$ is compact and hence
\begin{align}\label{unif-conv}
\int fd\mu_j\to\int fd\mu \;\text{ uniformly in }\; f\in S.
\end{align}

Now let $T,T_j\in\Pi_n$ be extremal polynomials in $L_p(\mu)$ and $L_p(\mu_j)$ respectively, that is, $t_{p,n}(\mu)=\|T\|_{L_p(\mu)}$ and $t_{p,n}(\mu_j)=\|T_j\|_{L_p(\mu_j)}$. Since the zeros of each $T_j$ lie on $D$, the coefficients of $T_j$'s are uniformly bounded in $j$ and hence there is a number $M$ such that $\|T_j\|_D\le M$ for all $j$. Thus, by the uniform convergence \eqref{unif-conv}, we have $\int|T_j|^pd\mu_j\to\int|T|^pd\mu$ and hence \eqref{t-mu-conv} holds. Finally, \eqref{W-mu-conv} follows from \eqref{t-mu-conv} and \eqref{W-def}.
\end{proof}

\begin{theorem}\label{Thm:Inf-W/S}
For any non-polar compact set $K\subset\bb R$ and any $p\in(0,\infty)$, $n\in\bb N$,
\begin{align}\label{Inf-W/S}
\inf_{\mu}\big[W_{p,n}(\mu)\big]^p/S(\mu)=1,
\end{align}
where the infimum is taken over measures $\mu$ with polynomial densities with respect to the equilibrium measure $\mu_K$, that is, $d\mu(z)=w(z)d\mu_K(z)$ with a polynomial $w(z)$ positive on $K$.
\end{theorem}
\begin{proof}
By shifting $K$ we may assume that $0\in K$ and that it is a regular point, that is, $g_K(z)\to0$ as $z\to0$. Consider $d\mu_\eps(z)=w_\eps(z)d\mu_K(z)$ with the weight $w_\eps(z)=(z^2+\eps^2)^{-np/2}$. Then, using $x^n$ as a trial polynomial, we obtain
\begin{align}
\big[t_{p,n}(\mu_\eps)\big]^p \le \int |x^n|^p w_\eps(x)d\mu_K(x) = \int \left|\frac{x^2}{x^2+\eps^2}\right|^{np/2} d\mu_K \le 1.
\end{align}
Using \eqref{gK-muK} and the fact that $0\in K$ is a regular point, we also get
\begin{align}
S(\mu_\eps)
&=\exp\left(-\frac{np}{2}\int\log\big(|z+i\eps||z-i\eps|\big)d\mu_K(z)\right)\no\\
&=\exp\left(-\frac{np}{2}\bigl(g_K(i\eps)+g_K(-i\eps)\bigr)\right)\ca(K)^{-np}
\to \ca(K)^{-np}
\end{align}
as $\eps\to0$. Thus,
\begin{align}\label{Inf-2}
\limsup_{\eps\to0}\big[W_{p,n}(\mu_\eps)\big]^p/S(\mu_\eps)\le1.
\end{align}

Next, for a fixed $\eps\in(0,1)$ we approximate $w_\eps$ by polynomials $w_j$ in the uniform norm on $K$. Since $w_\eps\ge c>0$ on $K$ we may assume $w_j>0$ on $K$ for each $j$ and
\begin{align}\label{Inf-UnifConv}
\Big\|1-\frac{w_j}{w_\eps}\Big\|_K \le \frac1c\|w_\eps-w_j\|_K\to0
\;\text{ as }\; j\to\infty.
\end{align}
Let $d\mu_j=w_jd\mu_K$, then $\mu_j\to\mu_\eps$ in the weak star sense and hence, by Theorem~\ref{Thm:mu-conv},
\begin{align}
\lim_{j\to\infty}W_{p,n}(\mu_j) = W_{p,n}(\mu_\eps)
\end{align}
and, by the uniform convergence \eqref{Inf-UnifConv},
\begin{align}
S(\mu_j)/S(\mu_\eps) = \exp\left[\int\log\frac{w_j(x)}{w_\eps(x)}d\mu_K\right] \to 1 \;\text{ as }\; j\to\infty.
\end{align}
Thus,
\begin{align}\label{Inf-3}
\lim_{j\to\infty}\big[W_{p,n}(\mu_j)\big]^p/S(\mu_j) = \big[W_{p,n}(\mu_\eps)\big]^p/S(\mu_\eps).
\end{align}
The theorem now follows from \eqref{LB-univ}, \eqref{Inf-2} and \eqref{Inf-3}.
\end{proof}

\section{Inverse polynomial images} \label{Sec:Preimages}

Throughout this section let $\cc T(z)=\tau z^N+\dots$, $\tau\neq0$, be a  polynomial of degree $N\geq1$ and $\mu_0$ be a finite Borel measure with compact support $K_0$ in $\bb C$. We will consider the polynomial pre-image set
\begin{align}\label{K-preim}
K:=\cc T^{-1}(K_0)
\end{align}
and a measure $\mu$ on $K$ defined via a certain pull-back procedure from $\mu_0$.
For this, let $\{\cc T_j^{-1}\}_{j=1}^N$ be a complete set of inverse branches of $\cc T$ and $\cc R(z)=\tau z^{N-1}+\dots$ be a polynomial of degree $N-1$ with the same leading coefficient as $\cc T$ and such that 
\begin{align}\label{TR-assumpt}
0<\cc R(z)/\cc T'(z)<\infty, \quad z\in K,
\end{align}
after possibly canceling the common zeros of $\cc R$ and $\cc T'$. For example, $\cc R=\cc T'/N$ satisfies these conditions.
Then (see (1.9) in \cite{PehSte00}) there exists a finite positive Borel measure $\mu$ with $\supp(\mu)=K$ such that
\begin{align}\label{mu-preim}
\int f(z)\,d\mu(z) = \sum_{j=1}^N\int  f(\cc T_j^{-1}(z))\,
\frac{\cc R(\cc T_j^{-1}(z))}{\cc T' (\cc T_j^{-1}(z))}\, d\mu_0(z),
\quad f\in C(K).
\end{align}
In fact, given $\cc T$, $\cc R$, and $\mu_0$ \eqref{mu-preim} uniquely determines $\mu$ and we define the linear map $\UTR(\mu_0):=\mu$. The identity \eqref{mu-preim} extends to all $f\in L_1(\mu)$. If $\cc R$ satisfies the above assumptions for a set $K_0$ then it automatically satisfies them for any subsets $K_1\subset K_0$ since $\cc T^{-1}(K_1)\subset\cc T^{-1}(K_0)$. Thus, in fact, the transformation $\UTR$ is well defined for all measures supported on $K_0$.
When $\cc R=\cc T'/N$ we will use the simplified notation $\UT:=\UTR$. By Lemma~4 in \cite{PehSte00}, the transformation $\UT$ maps the equilibrium measure on $K_0$ into the equilibrium measure on $K$,
\begin{align}\label{equil-preim}
\UT(\mu_{K_0})=\mu_{K}.
\end{align}

\begin{lemma}
Under the above assumptions on $\cc T$ and $\cc R$ we have for all $z\in K_0$,
\begin{align}\label{magic}
\sum_{j=1}^N \frac{\cc R(\cc T_j^{-1}(z))}{\cc T' (\cc T_j^{-1}(z))}=1.
\end{align}
Hence, the transformation $\mu=\UTR(\mu_0)$ preserves the total mass of $\mu_0$, that is, $\mu(K)=\mu_0(K_0)$.
\end{lemma}
\begin{proof}
Fix $z\in K_0$ and consider the partial fraction decomposition
\begin{align}
\frac{\cc R(y)}{\cc T(y)-z} = \sum_{j=1}^N \frac{A(\cc T_j^{-1}(z))}{y-\cc T_j^{-1}(z)}, \quad y\in\bb C\bs K,
\end{align}
where $A=\cc R/\cc T'$ which by assumption has no poles on $K$. Then letting $y\to\infty$ and comparing the $1/y$ terms yields \eqref{magic}. Taking $f\equiv 1$ in \eqref{mu-preim} and using \eqref{magic} implies $\mu(K)=\mu_0(K_0)$.
\end{proof}

\begin{theorem}\label{Thm:compo}
Let $d\mu_0=w\,d\mu_{K_0}+d \mu_{0,s}$ be the Lebesgue decomposition of $\mu_0$ with respect to $\mu_{K_0}$, that is, $w$ is the Radon--Nikodym derivative of $\mu_0$ with respect to $\mu_{K_0}$ and $\mu_{0,s}$ is the singular part of $\mu_0$ with respect to $\mu_{K_0}$. Then:
\begin{enumerate}[\rm (i)]
    \item $\mu=\UTR(\mu_0)$ has the following Lebesgue decomposition with respect to $\mu_K$,
    \begin{align}\label{compo}
    d\mu=\frac{N\mathcal{R}}{\mathcal{T}^\prime}\,w\circ\cc T\,d\mu_{K}+d\mu_{s}
    \end{align}
    where $\mu_{s}$ is the singular part of $\mu$ with respect to $\mu_{K}$.
    In addition, $\UTR$ maps absolutely continuous, singular continuous,
    pure point parts of $\mu_0$ to absolutely continuous, singular continuous, pure point parts of $\mu$ respectively.
    In particular,
    \begin{align}
    \UT(\mu_0)=w\circ\cc T\,d\mu_{K}+d\mu_{s}.
    \end{align}
    \item If $\mu=\UTR(\mu_0)$ then $S(\mu)=S_K(N\cc R/\cc T')S(\mu_0)$. In particular, if $\mu=\UT(\mu_0)$ then $S(\mu)=S(\mu_0)$.
\end{enumerate}
\end{theorem}
\begin{proof}
(i) First, we assume that $\mu_{0,s}=0$, that is, $d \mu_0= w d\mu_{K_0}$.

Since $\UT(\mu_{K_0})=\mu_{K}$, for any $g\in C(K)$,
\begin{align}
\int g(z)\,\frac{\mathcal{R}(z)}{\mathcal{T}^\prime (z)} w(\cc T(z))\,d\mu_{K}(z)= \int \frac1N \sum_{j=1}^N  \left(g(\mathcal{T}_j^{-1}(z)) \frac{\mathcal{R}(\mathcal{T}_j^{-1}(z))}{\mathcal{T}^\prime (\mathcal{T}_j^{-1}(z))}\right) w(z)\,d\mu_{K_0}(z).
\end{align}
is satisfied in view of \eqref{mu-preim} and \eqref{equil-preim}. This implies that for the measure $d\nu(z)=\frac{N\mathcal{R}(z)}{\mathcal{T}^\prime(z)}w(\cc T(z)) \, d\mu_{K}(z)$ and $g\in L_1(\nu)$,
\begin{align}
\int g(z)\,d\nu(z)= \int  \sum_{j=1}^N  \left(g(\mathcal{T}_j^{-1}(z)) \frac{\mathcal{R}(\mathcal{T}_j^{-1}(z))}{\mathcal{T}^\prime (\mathcal{T}_j^{-1}(z))}\right) w(z)\,d\mu_{K_0}(z).
\end{align}
Thus, $d\mu(z)=\frac{N\mathcal{R}(z)}{\mathcal{T}^\prime(z)}w(\cc T)d\mu_K(z)$.

Now, we assume that $\mu_{0,s}\neq 0$. Clearly $\UTR(d \mu_0)= \UTR(w d\mu_{K_0})+\UTR(d \mu_{s,0})$. Let $\nu_1:= \UTR(d \mu_{s,0})$
and $d\nu_1= h d\mu_{K}+d\nu_{1,s}$. Let $A$ be a Borel subset of $K_0$ such that $\mu_{K_0}(A)=1$ and $\mu_{s,0}(A)=0$. Then $\nu_1(\mathcal{T}^{-1}(A))=0$. This implies that $h=0$ $\mu_{K}$-a.e since
\begin{align}
0 = \int \cf_{\mathcal{T}^{-1}(A)}\, d\nu_1
\ge \int \cf_{\mathcal{T}^{-1}(A)}\, h\, d\mu_{K}
& = \int \frac1N \sum_{j=1}^N h(\mathcal{T}_j^{-1}(z))\, d\mu_{K_0}(z)
= \int h\, d\mu_{K}.
\end{align}
Hence, this proves \eqref{compo}. In particular \eqref{compo} implies that $\UTR$ maps absolutely continuous and singular parts of $\mu_0$ to absolutely continuous and singular parts of $\mu$ respectively. It follows from \eqref{mu-preim} that $\UTR$ maps singular continuous and pure point parts of $\mu_0$ to singular continuous, and pure point parts of $\mu$ respectively.\\

(ii) Let $\eps>0$. Since $\log(w(\cc T)+\eps)\in L_1(\mu_{K})$, it follows from \eqref{mu-preim}, \eqref{equil-preim} that
\begin{align}
\int \log(w(\cc T(z))+\eps)\, d\mu_{K}(z)= \int \log(w(z)+\eps)\, d\mu_{K_0}(z).
\end{align}
Letting $\eps\downarrow 0$ and using a simple argument involving the monotone convergence theorem we get
\begin{align}\label{super}
\int \log(w(\cc T(z)))\, d\mu_{K}(z)= \int \log(w(z))\, d\mu_{K_0}(z).
\end{align}
Thus by \eqref{compo} and \eqref{super} we get  $S(\mu)=S_K(N\cc R/\cc T')S(\mu_0)$. In particular, when $\mu=\UT(\mu_0)$ we have $S(\mu)=S(\mu_0)$.
\end{proof}

In the following we will assume that $K_0$ and hence also $K=\cc T^{-1}(K_0)$ have positive capacity. The next result shows that the Widom factors are invariant under the transformation $\UTR$.

\begin{theorem}\label{Thm:W-invar}
Let $\cc T(z)=\tau z^N+\dots$ be a degree $N\geq1$ polynomial, $p\in [1,\infty)$, $\mu_0$ be a finite Borel measure, and $\{T_n\}_{n=1}^\infty$ be monic extremal polynomials in $L_p(\mu_0)$ with $\deg(T_n)=n$. Then
\begin{align} \label{EP-invar}
S_{nN}(z)=\tau^{-n}\,T_n(\cc T(z)), \quad n\in\bb N,
\end{align}
are monic extremal polynomials in $L_p(\mu)$ for $\mu=\UTR(\mu_0)$
and the corresponding Widom factors satisfy
\begin{align}\label{W-invar}
W_{p,nN}(\mu) = W_{p,n}(\mu_0), \quad n\in\bb N.
\end{align}
\end{theorem}
\begin{proof}
The identity \eqref{EP-invar} follows from \cite[Theorem~3]{PehSte00}. Then, by \eqref{mu-preim} and \eqref{magic},
\begin{align}
\|S_{nN}\|_{L_p(\mu)}^p &= |\tau|^{-np} \int |T_n(\cc T(z))|^p\, d\mu(z) \no
\\
&= |\tau|^{-np} \sum_{j=1}^N \int \frac{\cc R(\cc T_j^{-1}(z))}{\cc T^\prime (\cc T_j^{-1}(z))}\, |T_n(z)|^p\, d\mu_{{0}}(z) \no
\\
&= |\tau|^{-np} \int |T_n(z)|^p\, d\mu_{{0}}(z)
= |\tau|^{-np}\, \|T_n\|_{L_p(\mu_0)}^p.
\end{align}
Let $K_0=\supp(\mu_0)$, then $K=\cc T^{-1}(K_0)=\supp(\mu)$ and, by \cite[Theorem\,5.2.5]{Ran95}, we have $\ca(K)^N=\ca(K_0)/|\tau|$.
Therefore,
\begin{align}
W_{p,nN}(\mu)
= \frac{\|S_{nN}\|_{L_p(\mu)}}{\ca(K)^{nN}}
= \frac{\|T_n\|_{L_p(\mu_0)}/|\tau|^{n}}{(\ca(K_0)/|\tau|)^{n}}
= W_{p,n}(\mu_0).
\end{align}
\end{proof}

Next, we apply the above theorem to reflectionless measures on finite gap sets. A set $K\subset\bb R$ is called finite gap if $K=\bigcup_{k=1}^{\ell+1}[a_k,b_k]$ with $a_1<b_1<a_2<\dots<b_{\ell+1}$. The class of reflectionless measures on $K$ consists of absolutely continuous measures $\mu$ of the form
\begin{align}\label{refl-meas}
d\mu(x) &= \frac{\cf_K(x)}{\pi|(x-a_1)(x-b_\ell)|^{1/2}}
\prod_{k=1}^{\ell}\frac{|x-d_k|}{|(x-b_k)(x-a_{k+1})|^{1/2}}\,dx
=\prod_{k=1}^{\ell}\left|\frac{x-d_k}{x-c_k}\right| d\mu_K(x),
\end{align}
where $d_k\in[b_k,a_{k+1}]$, $k=1,\dots,\ell$, are arbitrary points in gaps and $\{c_k\}_{k=1}^{\ell}$ are the critical points of the Green function $g_K$ on $\bb R\bs K$. The equilibrium measure $\mu_K$ is a representative of this class corresponding to the choice $d_k=c_k$, $k=1,\dots,\ell$.  If $K\subset\bb R$ is a polynomial pre-image $K=\cc T^{-1}([-1,1])$ then the critical points $\{c_k\}_{k=1}^{\ell}$ are the zeros of $\cc T'$ outside of $K$.

The following result provides a partial resolution to the open problems~1 and~2 in \cite{AlpZin20}, namely it establishes strengthened versions of the conjectured inequalities for a subsequence of Widom factors.

\begin{corollary} \label{Thm:W-refl}
Let $K\subset\bb R$ be a polynomial pre-image $K=\cc T^{-1}([-1,1])$ for some polynomial $\cc T$ of degree $N$. Then for each reflectionless measure $\mu$ on $K$ and, in particular, the equilibrium measure of $K$,
\begin{align}\label{W-refl}
\bigl[W_{p,nN}(\mu)\bigr]^p = \frac{2^p}{\sqrt\pi}\frac{\Gamma(\tfrac{p+1}2)}{\Gamma(\tfrac{p}2+1)},
\quad n\in\bb N, \; p\geq 1.
\end{align}
\end{corollary}
\begin{proof}
First suppose that $\mu$ is a reflectionless measure \eqref{refl-meas} with $d_k\in(b_k,a_{k+1})$ for all $k=1,\dots,\ell$ and define the polynomial
\begin{align}
\cc R(x)=\frac{1}{N}\cc T'(x)\prod_{k=1}^{\ell}\frac{x-d_k}{x-c_k}.
\end{align}
Then \eqref{TR-assumpt} holds and $\cc R$ satisfies the assumptions stated at the beginning of this section. Let $\mu_0$ be the equilibrium measure of $K_0=[-1,1]$, that is, $d\mu_0(x)=\frac{\cf_{K_0}\,dx}{\pi\sqrt{1-x^2}}$. Then, by \eqref{compo}, we have $\mu=\UTR(\mu_0)$.

By (6.1) in \cite{AlpZin20}, $[W_{p,n}(\mu_0)]^p=$\,RHS of \eqref{W-refl} for all $n\in\bb N$. Then \eqref{W-refl} for the reflectionless measure $\mu=\UTR(\mu_0)$ follows from \eqref{W-invar}.

Finally, to get \eqref{W-refl} for a general reflectionless measure $\mu$ (i.e., with some $d_k$'s at the gap edges) we approximate $\mu$ by the special reflectionless measures considered above (i.e., with all $d_k$'s lying inside the gaps) and apply Theorem~\ref{Thm:mu-conv}.
\end{proof}

In the next result we show that in the special case of $\mu_0$ supported on the unit circle, $\cc T(z)=z^N$, and $\mu=\UT(\mu_0)$ the entire sequence of extremal polynomials in $L_p(\mu)$ and the corresponding Widom factors can be obtained.

As a preliminary we recall a characterization for extremal polynomials. Let  $\mu$ be a finite Borel measure with compact support in $\bb C$ and $p\in[1,\infty)$. Then (see e.g., p.~51 in \cite{BorErd95}) a monic polynomial $S_n$ of degree $n$ is an extremal polynomial for $L_p(\mu)$ if and only if
\begin{align}\label{EP-char}
\int z^k |S_n(z)|^{p-2}\, \overline{S_n(z)}\, d\mu(z)=0, \quad k=0,1,\ldots,n-1.
\end{align}


\begin{theorem}\label{Thm:circle}
Let $\mathcal{T}(z)=z^N$ with $N\geq 2$, $p\in [1,\infty)$, $\mu_0$ be a finite Borel measure  supported on the unit circle, and $\{T_n\}_{n=0}^\infty$ be monic extremal polynomials in $L_p(\mu_0)$ with $\deg(T_n)=n$. Then
\begin{align}\label{EP-circ}
S_{\ell+nN}(z)= z^\ell\, T_n(z^N), \quad \ell\in\{0,\ldots,N-1\}, \; n\in\bb N_0,
\end{align}
are monic extremal polynomials in $L_p(\mu)$ for $\mu=\UT(\mu_0)$
and the corresponding Widom factors satisfy
\begin{align}\label{W-circ}
W_{p,\ell+nN}(\mu)= \ca(K)^{-\ell}\, W_{p,n}(\mu_0), \quad \ell\in\{0,\ldots,N-1\}, \; n\in\bb N_0.
\end{align}
\end{theorem}
\begin{proof}
It suffices to verify \eqref{EP-char} for the polynomials in \eqref{EP-circ}. The case $S_0\equiv 1$ is trivial. Assume $\ell+nN\geq1$, fix  $k\in\{0,\ldots,\ell+nN-1\}$, and let $\ti\ell\in\{0,\ldots,N-1\}$ and $\ti n\in\bb N_0$ be such that $k=\ti\ell+\ti nN$. Since $\mu_0$ is supported on $\pd\bb D$ so is $\mu$ by \eqref{K-preim}. Then, by \eqref{EP-circ} and \eqref{mu-preim}, we have
\begin{align}
&\int z^k\, |S_{\ell+nN}(z)|^{p-2}\, \overline{S_{\ell+nN}(z)}\, d\mu(z) =
\int z^{k-\ell}\, |T_n(z^N)|^{p-2}\, \overline{T_n(z^N)}\, d\mu{(z)} \no
\\ &\qquad=
\int z^{\ti\ell-\ell}\, \big(z^N\big)^{\ti n}\, |T_n(z^N)|^{p-2}\, \overline{T_n(z^N)}\, d\mu{(z)} \no
\\ &\qquad=
\int \frac1N\sum_{j=1}^N \big(\mathcal{T}_j^{-1}(z)\big)^{\ti\ell-\ell}\,
z^{\ti n}\, |T_n(z)|^{p-2}\,\overline{T_n(z)}\, d\mu_0(z). \label{circ1}
\end{align}
For each fixed $z\in\pd\bb D$, the $N$ pre-images $\cc T_j^{-1}(z)$ are some equispaced points on the unit circle, say $e^{i\te}e^{2\pi i j/N}$, $j=1,\ldots,N$. Then if $\ti\ell\neq\ell$, we have $0<|\ti\ell-\ell|\le N-1$ hence $e^{2\pi i(\ti\ell-\ell)/N}\neq1$ and we obtain
\begin{align}
\sum_{j=1}^N \big(\cc T_j^{-1}(z)\big)^{\ti\ell-\ell} = e^{i\te(\ti\ell-\ell)} \sum_{j=1}^N \big(e^{2\pi i(\ti\ell-\ell)/N}\big)^j = 0.
\end{align}
Thus, if $\ti\ell\neq\ell$, the RHS of \eqref{circ1} is zero. If $\ti\ell=\ell$ the RHS of \eqref{circ1} becomes
\begin{align} \label{circ2}
\int z^{\ti n}\, |T_n(z)|^{p-2}\,\overline{T_n(z)}\, d\mu_0(z)
\end{align}
and since $k=\ti\ell+\ti nN\le\ell+nN-1$ we have $\ti n\le n-1$. Then the expression in \eqref{circ2} is zero by \eqref{EP-char} since by assumption $T_n$ is an extremal polynomial in $L_p(\mu_0)$. Thus, the RHS of \eqref{circ1} is zero in either case and hence $S_{\ell+nN}(z)$ satisfies \eqref{EP-char} for all $k\in\{0,\ldots,\ell+nN-1\}$ and so is an extremal polynomial in $L_p(\mu)$.

The identity \eqref{W-circ} for $\ell=0$ is a special case of \eqref{W-invar} and for $\ell>0$ it follows from
$\|S_{\ell+nN}\|_{L_p(\mu)} = \|S_{nN}\|_{L_p(\mu)}$
which is a consequence of \eqref{EP-circ}.
\end{proof}

\section{Extremal sets} \label{Sec:Saturation}

In this section we investigate for which sets the lower bounds \eqref{LB-equil-C} and \eqref{LB-equil-R} become saturated.

For a compact set $K\subset\bb C$, the boundary $\pd\Om_K$ is called the outer boundary of $K$ and will be denoted by $O\pd(K)$. It is known \cite[Theorem~3.7.6]{Ran95} that the equilibrium measure $\mu_K$ for a set $K$ is supported on the outer boundary of $K$, that is, $\supp(\mu_K)\subset O\pd(K)$. In the next lemma we show that for regular sets $K$ the support of the equilibrium measure $\mu_K$ is equal to the outer boundary of $K$.

\begin{lemma}\label{Lem:muK-supp}
If $K\subset\bb C$ is a compact regular set then $\supp(\mu_K)=O\pd(K)$.
\end{lemma}
\begin{proof}
By the regularity assumption the Green function $g_K=0$ on $\pd\Om_K$. Suppose by contradiction that there exists $z_0\in\pd\Om_K\bs\supp(\mu_K)$. Then $g_K(z_0)=0$. Since $g_K\ge0$ and $g_K$ is harmonic on $\bb C\bs\supp(\mu_K)$ it follows from the minimum principle for harmonic functions that $g_K$ is identically zero on the component of $\bb C\bs\supp(\mu_K)$ containing $z_0$ and, in particular, on $\Om_K$, which is a contradiction. Thus, $\supp(\mu_K)=\pd\Om_K=O\pd(K)$.
\end{proof}

Now we can characterize the sets $K$ for which the lower bound $W_{p,n}(\mu_K)\ge1$ is saturated in terms of the outer boundary of $K$.

\begin{theorem} \label{Thm:LB-C-sat}
Let $K\subset\bb C$ be a compact regular set and $\mu_K$ be the equilibrium measure on $K$. Then $W_{p,n}(\mu_K)=1$ for some $p\in(0,\infty)$ and $n\in\bb N$ if and only if $O\pd(K)=Q_n^{-1}(\pd\bb D)$ for some polynomial $Q_n$ of degree $n$. In addition, in this case $W_{p,kn}(\mu_K)=1$ for all $p\in(0,\infty)$ and $k\in\bb N$.
\end{theorem}
\begin{proof}
Suppose the outer boundary of $K$ is a polynomial pre-image of the unit circle, that is, $O\pd(K)=\{z\in\bb C:|Q_n(z)|=1\}$ for some polynomial $Q_n(z)=c z^n+\dots$, $c\neq0$. In this case, $\ca(K)^n=\ca(O\pd(K))^n=1/|c|$. Now consider a trial monic polynomial $P_n=c^{-1}Q_n$. Since $\mu_K$ is supported on $O\pd(K)$ and $|Q_n|=1$ on $O\pd(K)$ we have
\begin{align}
\|P_n^k\|_{L_p(\mu_K)} = |c|^{-k} \|Q_n^k\|_{L_p(\mu_K)} =
|c|^{-k} = \ca(K)^{nk}.
\end{align}
Thus, $W_{p,kn}(\mu_K)\le1$ hence, by \eqref{LB-equil-C}, $W_{p,kn}(\mu_K)=1$ for all $p\in(0,\infty)$ and $k\in\bb N$.

Conversely, suppose $W_{p,n}(\mu_K)=1$ for some $p\in(0,\infty)$ and $n\in\bb N$. Let $P_n$ be a monic extremal polynomial of degree $n$ in $L_p(\mu_K)$, that is, $\|P_n\|_{L_p(\mu_K)}=\ca(K)^n$. By Jensen's inequality and Frostman's theorem, we have
\begin{align}
\|P_n\|_{L_p(\mu_K)}^p &= \exp\left[\log\left(\int |P_n|^p\, d\mu_K\right)\right]
\geq \exp\left[\int\log|P_n|^p\, d\mu_K\right]
\geq \ca(K)^{np},
\end{align}
hence the condition $\|P_n\|_{L_p(\mu_K)}=\ca(K)^n$ implies equality in Jensen's inequality,
\begin{align}
\log\int|P_n(z)|^p d\mu_K(z) = \int\log|P_n(z)|^p d\mu_K(z)
\end{align}
which holds if and only if $|P_n(z)|$ is constant $\mu_K$-a.e. By Lemma~\ref{Lem:muK-supp}, $\supp(\mu_K)=O\pd(K)$ and hence $|P_n(z)|=\ca(K)^n$ for all $z\in O\pd(K)$. Thus, by the maximum principle,
\begin{align}
\|P_n\|_K = \ca(K)^n,
\end{align}
and hence,  by \cite[Theorem~1.2]{CSZ3}, there exists a polynomial $Q_n$ of degree $n$ such that $O\pd(K)=Q_n^{-1}(\pd\bb D)$.
\end{proof}

Theorem~\ref{Thm:LB-C-sat} combined with \cite[Theorem~1.2]{CSZ3} implies the following result:

\begin{corollary}
Let $K$ be a regular compact subset of $\bb C$, $n\in \bb N$, and $p\in (0,\infty)$. Then $W_{\infty,n}(K)=1$ if and only if $W_{p,n}(\mu_K)=1$. If equalities hold then $n$-th monic extremal polynomial in $L_p(\mu_K)$ is the $n$-th monic Chebyshev polynomial on $K$.
\end{corollary}

Next, we characterize the sets $K\subset\bb R$ for which the lower bound $W_{2,n}(\mu_K)\ge\sqrt2$ is saturated.

\begin{theorem} \label{Thm:LB-R-sat}
Let $K\subset\bb R$ be a regular compact set. Then $W_{2,n}(\mu_K)=\sqrt{2}$ if and only if $K=Q_n^{-1}\big([-1,1]\big)$ for some polynomial $Q_n$ of degree $n$.
\end{theorem}
\begin{proof}
If $K=Q_n^{-1}\big([-1,1]\big)\subset\bb R$, then $W_{2,n}(\mu_K)=\sqrt{2}$ by Corollary~\ref{Thm:W-refl}.

Next, suppose $K\subset\bb R$ is a regular compact set such that $W_{2,n}(\mu_K)=\sqrt{2}$.
Let $\x:\bb D\to\ol{\bb C}\bs K$ be the universal covering map of $\overline{\bb C}\bs K$ normalized by $\x(0)=\infty$ and $\lim_{z\to0}z\x(z)>0$. Then $\x$ is a meromorphic function of bounded characteristic (see Theorem~1, Section~5.1, Chapter~7 in \cite{Nev70}) and it has non-tangential boundary values a.e.\ on $\pd\bb D$ (see Theorem~2, Section~5.4, Chapter~7 in \cite{Nev70} or \cite[Theorem~V.9]{Tsuji}) and $\x(e^{i\te})\in K$ a.e.\ (see Section~5.5, Chapter~7 in \cite{Nev70}), hence $\x_{\restriction_{\pd\bb D}}\in L_\infty(d\te)$. This implies (see Section~2.4 in \cite{Fish}) that
\begin{align}
\int f d\mu_K = \int f(\x({e^{i\te}}))\, \frac{d\te}{2\pi},
\quad f\in L_1(\mu_K).
\end{align}

Let $\Gamma$ be the Fuchsian group of Mobius transformations (see Section~9.5 in \cite{Sim11}) on $\bb D$ so that $\x(z)= \x(w)$ if and only if there is a $\gamma\in \Gamma$ with $z=\gamma(w)$. Define the Blaschke product $B$ by
\begin{align}
B(z)= \prod_{\gamma\in\Gamma} \frac{|\gamma(0)|}{\gamma(0)} \gamma{(z)},
\quad z\in \bb D.
\end{align}
It is known (see e.g.\ \cite[Theorem~16.11]{Mar19} or \cite{Pom76}) that $B(z)$ is an analytic function with $|B(e^{i\te})|=1$ a.e.\ on $\pd\bb D$, simple zeros at $\x^{-1}(\infty)$, and
\begin{align}\label{grin}
|B(z)|= e^{-g_K(\x(z))}.
\end{align}
Then, in particular, we have (cf.\ (9.7.35) and (9.7.37) in \cite{Sim11})
\begin{align}\label{xBcap}
\lim_{z\to0}\x(z)B(z)=\ca(K).
\end{align}

Let $P_n$ be the $n$-th monic orthogonal polynomial for $\mu_K$. Since $K\subset \mathbb{R}$, the polynomial $P_n$ is real. The function $B(z)^n P_n(\x(z))$ has only removable singularities and therefore it can be identified with a bounded analytic function on $\bb D$ such that $\lim_{z\to0}B(z)^n P_n(\x(z))=\ca(K)^n$. Since $P_n(\x(e^{i\te}))\in\bb R$ for a.e.\ $\te$ we have, as in the proof of \cite[Theorem~3.1]{AlpZin20},
\begin{align}
2C(K)^n &= \int_0^{2\pi} P_n(\x(e^{i\te}))\big(B(e^{i\te})^n+\ol{B(e^{i\te})^n}\,\big)\, \f{d\te}{2\pi}.
\end{align}
Then, by Cauchy--Schwarz inequality,
\begin{align}\label{CS-ineq}
2C(K)^n &\le \left[\int_0^{2\pi} P_n(\x(e^{i\te}))^2\, \f{d\te}{2\pi}\right]^{\f12}
\left[\int_0^{2\pi}\big(B(e^{i\te})^n+\ol{B(e^{i\te})^n}\,\big)^{2}\, \f{d\te}{2\pi}\right]^{\f12}
\\
&= \|P_n\|_{L_2(\mu_K)} \left[\int_0^{2\pi} 2+2\Re\big(B(e^{i\te})^{2n}\big)\, \f{d\te}{2\pi}\right]^{\f1{2}}
\\
&= \|P_n\|_{L_2(\mu_K)} \Big[2+2\Re\big(B(0)^{2n}\big)\Big]^{\f1{2}}
= \sqrt{2}\, \|P_n\|_{L_2(\mu_K)},
\end{align}
hence $\|P_n\|_{L_2(\mu_K)}\ge\sqrt{2}\,\ca{(K)}^n$.
Since, by assumption, $\|P_n\|_2=\sqrt{2}\,\ca{(K)}^n$ we have equality in the Cauchy--Schwarz inequality \eqref{CS-ineq} which occurs only when the two functions are proportional hence for some constant $c$ and a.e.\ $\te$,
\begin{align}\label{CS-eq}
|P_n(\x(e^{i\te}))|^2 = c \big(B(e^{i\te})^n+\ol{B(e^{i\te})^{n}}\,\big)^{2}=c \big(B(e^{i\te})^n+{B(e^{i\te})^{-n}}\,\big)^{2}.
\end{align}
Let $F(z):=P_n(\x(z))/\big(B(z)^n+{B(z)^{-n}}\,\big)=P_n(\x(z)) B(z)^n/(B(z)^{2n}+1)$ and note that the only singularities of $F$ on $\bb D$ are removable. Since $P_n(\x(e^{i\te}))\in \bb R$ for a.e.\ $\te$ it follows from \eqref{CS-eq} that $F(e^{i\te})^2=c$ for a.e.\ $\te$ and hence $F^2$ is identically constant on $\bb D$. By \eqref{xBcap}, $\lim_{z\to0}F(z)=\ca(K)^n$ hence $F\equiv\ca(K)^n$ and so
\begin{align}\label{bz}
P_n(\x(z)) = \ca(K)^n\big(B(z)^n+B(z)^{-n}\big), \quad z\in \bb D.
\end{align}
Hence, we have, by \eqref{grin}, \eqref{bz} that
\begin{align}\label{ident}
|P_n(z)| \le \ca(K)^n\big(e^{-ng_K(z)}+e^{ng_K(z)}\big), \quad z \in \bb C\bs K.
\end{align}
Since $K$ is regular \eqref{ident} implies
$\|P_n\|_K \le 2\ca(K)^n$ which combined with \eqref{LB-cheb-R} yields
\begin{align}
\|P_n\|_K = 2\ca(K)^n.
\end{align}
Then, by \cite[Theorem~1.1]{CSZ3} or \cite[Theorem~1]{Tot11}, there exists a polynomial $Q_n$ of degree $n$ such that $K=Q_n^{-1}\big([-1,1]\big)$.
\end{proof}

Theorem~\ref{Thm:LB-R-sat} and \cite[Theorem~1.1]{CSZ3} lead to the following corollary:
\begin{corollary}
Let $K\subset\bb R$ be a regular compact set and $n\in\bb N$. Then $W_{\infty,n}(K)=2$ if and only if $[W_{2,n}(\mu_K)]^2=2$. If equalities hold then the $n$-th monic extremal polynomials in $L_\infty(K)$ and $L_2(\mu_K)$ are the same.
\end{corollary}

\section{Widom factors for the equilibrium measure on a circular arc}
\label{Sec:Arc}
Let $\mu$ be a unit Borel measure such that $\mathrm{supp}(\mu)\subset \partial \bb D$ and $\mathrm{supp}(\mu)$ contains infinitely many points. Then for each $n\in \bb N_0$, the $n$-th Verblunsky coefficient $\al_n$ (or $\al_n(\mu)$) is defined by (see \cite[Eq.~(1.5.20)]{Sim05}) $\al_n= -\overline{\Phi_{n+1}(0)}$ where $\Phi_{n+1}$ is the monic orthogonal polynomial for $\mu$ of degree $n+1$.

In this section we consider orthogonal polynomials on the circular arc
\begin{align}
K_\ga = \big\{e^{i\te}:\te\in[\pi-\ga,\pi+\ga]\big\}, \quad 0<\ga<\pi.
\end{align}
In the case of the uniform measure $d\mu_\ga=\cf_{K_\ga}(e^{i\te})\f{d\te}{2\ga}$ on $K_\ga$, the orthogonal polynomials have been investigated in \cite{Mag06} in connection with the Gr\"unbaum--Delsarte--Janssen--Vries problem. One of the main results of \cite{Mag06}, stated in terms of the Widom factors, asserts that the sequence $\{W_{2,n}(\mu_\ga)\}_{n=1}^\infty$ is strictly monotone increasing. In addition, it is shown in \cite{Mag06} that the Verblunsky coefficients are negative and their absolute values are decreasing with $\ga$.
Below we extend these results to the equilibrium measure on $K_\ga$.

Recall that (e.g., \cite[Eq.~(9.7.13)]{Sim05} or \cite[Section~3]{NT13}) the equilibrium measure $\mu_{K_\ga}$ is given by $d\mu_{K_\ga}(e^{i\te}) = \cf_{K_\ga}(e^{i\te})w(\te)\f{d\te}{2\pi}$ where
\begin{align}\label{Arc-w}
w(\te) = \f{\sin(\te/2)}{\sqrt{\cos^2((\pi-\ga)/2)-\cos^2(\te/2)}} = \f{\sin(\te/2)}{\sqrt{\sin^2(\ga/2)-\cos^2(\te/2)}}.
\end{align}

\begin{theorem} \label{Thm:Arc}
    Let $0<\ga<\pi$ and  $0<\ga_1<\ga_2<\pi$.
\begin{enumerate}[\rm (i)]
\item The sequence $\big\{W_{2,n}(\mu_{K_\ga})\big\}_{n=1}^\infty$ is strictly monotone increasing with $n$,
    \begin{align} \label{Arc-lim}
    \lim_{n\to\infty} \big[W_{2,n}(\mu_{K_\ga})\big]^2 = 1+\cos(\ga/2),
    \end{align}
    and
    \begin{align} \label{Arc-inf-sup}
    \inf_{n\in\bb N} \big[W_{2,n}(\mu_{K_\ga})\big]^2 = 1+\cos^2(\ga/2),
    \quad
    \sup_{n\in\bb N} \big[W_{2,n}(\mu_{K_\ga})\big]^2 = 1+\cos(\ga/2).
    \end{align}
\item The Verblunsky coefficients $\al_n(\mu_{K_\ga})$ are negative and their absolute values are strictly monotone decreasing with $\ga$,
    \begin{align} \label{Arc-al-ga-mon}
    |\al_n(\mu_{K_{\ga_1}})|>|\al_n(\mu_{K_{\ga_2}})|, \quad n\in\bb N_0.
    \end{align}
\item The Widom factors $W_{2,n}(\mu_{K_\ga})$ are strictly monotone decreasing with $\ga$,
    \begin{align} \label{Arc-W-ga-mon}
    W_{2,n}(\mu_{K_{\ga_1}})>W_{2,n}(\mu_{K_{\ga_2}}), \quad n\in \bb N.
    \end{align}
\end{enumerate}
\end{theorem}
\begin{proof}
(i) It is a result of Widom \cite[Theorem~12.3 and~6.2]{Wid69} that the sequence $W_{2,n}(\mu_{K_\ga})$ has a limit and
\begin{align} \label{Widom-lim}
\lim_{n\to\infty} \big[W_{2,n}(\mu_{K_\ga})\big]^2 = 2\pi R(\infty)\ca(K_\ga),
\end{align}
where $R(z)$ is the unique non-vanishing analytic function on $\ol{\bb C}\bs K_\ga$ with $R(\infty)>0$ and boundary values satisfying $|R(e^{i\te})|=\f1{2\pi}w(\te)$ on $K_\ga$.
Consider the function
\begin{align}
F(z) = \f{z-1}{\sqrt{(z+1)^2-4\sin^2(\ga/2)z}}
\end{align}
with the branch of the square root chosen such that $F(\infty)=1$. Then
\begin{align}
\big|F(e^{i\te})\big| = \left| \f{e^{i\te/2}-e^{-i\te/2}}{\sqrt{(e^{i\te/2}+e^{-i\te/2})^2-4\sin^2(\ga/2)}}
\right|
= w(\te)
\end{align}
and $F$ is analytic on $\ol{\bb C}\bs K_\ga$ with a single simple zero at $z=1$. To remove this zero, consider the function
\begin{align}
B(z) = \f{\cos(\ga/2)}{\sin(\ga/2)} \left( \sqrt{\left(\f{z+1}{z-1}\right)^2+\left(\f{\sin(\ga/2)}{\cos(\ga/2)}\right)^2} - \f{z+1}{z-1}\right)
\end{align}
with the branch of the square root chosen such that $B$ has a zero at $z=1$. Then $B(z)$ has no other zeros, is analytic on $\ol{\bb C}\bs K_\ga$, satisfies $|B|=1$ on $K_\ga$, and
\begin{align}
B(\infty) = \f{1-\cos(\ga/2)}{\sin(\ga/2)}
= \f{\sin(\ga/2)}{1+\cos(\ga/2)}.
\end{align}
It follows that $R(z)=\f{1}{2\pi}F(z)/B(z)$ and hence
\begin{align} \label{Arc-Rinfty}
R(\infty) = \f{1}{2\pi} \f{1+\cos(\ga/2)}{\sin(\ga/2)}.
\end{align}
In addition, by \cite[Table~5.1]{Ran95},
\begin{align} \label{Arc-cap}
\ca(K_\ga)=\sin(\ga/2).
\end{align}
Thus, \eqref{Arc-lim} follows from \eqref{Widom-lim}, \eqref{Arc-Rinfty}, and \eqref{Arc-cap}.

Next, we show that the sequence $\{W_{2,n}(\mu_{K_\ga})\}_{n=1}^\infty$ is strictly monotone increasing. We start by expressing the Widom factors in terms of the Verblunsky coefficients $\{\al_k\}_{k=0}^\infty$ of the measure of orthogonality $\mu_{K_\ga}$,
\begin{align}\label{W-al}
W_{2,n}(\mu_{K_\ga}) = \ca(K_\ga)^{-n}\prod_{j=0}^{n-1}\sqrt{1-|\al_j|^2}.
\end{align}
Since, by \eqref{Widom-lim}, $W_{2,n+1}(\mu_{K_\ga})/W_{2,n}(\mu_{K_\ga})\to1$ it follows from \eqref{W-al} and \eqref{Arc-cap} that
\begin{align}\label{al-lim}
\lim_{n\to\infty}|\al_n| = \cos(\ga/2).
\end{align}
In addition, by \eqref{W-al}, $W_{2,n+1}(\mu_{K_\ga})>W_{2,n}(\mu_{K_\ga})$ is equivalent to
\begin{equation} \label{al-ineq}
|\al_n|<\cos(\ga/2).
\end{equation}
To show this inequality we will use the Szeg\H{o} mapping and the inverse Geronimus relations \cite[Section~13.1]{Sim05}. Let $\mu=Sz(\mu_{K_\ga})$, that is, $d\mu(x)=f(x)dx$, where $f(x)$ satisfies $w(\te)=\pi\sqrt{4-x^2}f(x)$ with $x=2\cos\te$. Using $\cos^2(\te/2)=\f14(2+x)$ and $\sin^2(\te/2)=\f14(2-x)$ we get from \eqref{Arc-w} that
\begin{align}
w(\te) = \f{\sqrt{2-x}}{\sqrt{-2\cos(\ga)-x}} = \f{\sqrt{4-x^2}}{\sqrt{(-2\cos(\ga)-x)(2+x)}}.
\end{align}
Setting $c:=-2\cos(\ga)$ and $L_\ga:=[-2,c]$ then yields
\begin{align}
d\mu(x) = \f{1}{\pi}\f{\cf_{L_\ga}(x)}{\sqrt{(c-x)(2+x)}}\,dx
\end{align}
hence $\mu=\mu_{L_\ga}$, the equilibrium measure of the interval $L_\ga$. The monic orthogonal polynomials $P_n$ in $L_2(\mu_{L_\ga})$ are just the classical Chebyshev polynomials of the first kind, $T_n\big(\tfrac12(z+z^{-1})\big)=\tfrac12(z^n+z^{-n})$, appropriately shifted and rescaled,
\begin{align}\label{PnTn}
P_n(x)= 2a^n\,T_n\left(\f{x-b}{2a}\right), \quad n\in\bb N,
\end{align}
where $a=(c+2)/4=\sin^2(\ga/2)$ and $b=(c-2)/2$. For convenience of notation we set $P_0(x):=2$ which is compatible with \eqref{PnTn} for $n=0$ since $T_0(x)\equiv1$.
Then, by the inverse Geronimus relations \cite[Theorem~13.1.10]{Sim05},
\begin{align}\label{invGer}
\al_{2k}=-\frac{u_k-v_k}{u_k+v_k}, \quad
\al_{2k+1}= 1-\tfrac12(u_{k+1}+v_{k+1}), \quad k\in\bb N_0,
\end{align}
where
\begin{align}
u_k &= \frac{P_{k+1}(2)}{P_k(2)} = a\,\frac{T_{k+1}((6-c)/(2+c))}{T_{k}((6-c)/(2+c))}>0,\label{nr1}\\
v_k &= -\frac{P_{k+1}(-2)}{P_{k}(-2)} = -a\,\frac{T_{k+1}(-1)}{T_k(-1)} = a>0,
\quad k\in\bb N_0.\label{nr2}
\end{align}
Note that particular choice of $P_0$ affects only $u_0$ and $v_0$ but does not affect the value of $\al_0$ in \eqref{invGer}. Introduce a new variable
\begin{align} \label{s-c-def}
s := \frac{6-c}{2+c} = \frac{2}{\sin^2(\ga/2)}-1.
\end{align}
Since $s>1$ we obtain from
\begin{align}
T_k(s)= \tfrac12 \left[\big(s+\sqrt{s^2-1}\,\big)^k + \big(s-\sqrt{s^2-1}\,\big)^k\right]
\end{align}
that the $u_k$'s strictly increase as $k$ increases and $u_0=2-a$.
Then since $0<a<1$ it follows from \eqref{invGer} that both subsequences $\al_{2k}$ and $\al_{2k+1}$ are negative and strictly decreasing.
This strict monotonicity combined with the limit \eqref{al-lim} then yields \eqref{al-ineq} and hence the Widom factors $W_{2,n}(\mu_{K_\ga})$ are strictly monotone increasing.

Finally, since $u_0=2-a$ and $v_0=a$ we have $\al_0=a-1=-\cos^2(\ga/2)$ hence $1-|\al_0|^2=\sin^2(\ga/2)(1+\cos^2(\ga/2))$ and so $[W_{2,1}(\mu_{K_\ga})]^2=1+\cos^2(\ga/2)$. Combined with the monotonicity and the limit of $[W_{2,n}(\mu_{K_\ga})]^2$ this yields \eqref{Arc-inf-sup}.\\

(ii) Note that $s$ increases as $\ga$ decreases. Hence it is enough to show that $|\al_n|$ strictly increases as $s$ increases. Since $\al_n$ is negative we have to show that $-\al_n$ strictly increases as $s$ increases. We consider the cases of $n$ even and odd separately.

First, let $n=2k$ for some $k\in \bb N_0$. Then by \eqref{invGer}, \eqref{nr1} and \eqref{nr2}, we have
\begin{align} \label{al-even}
-\al_n=\left(\frac{T_{k+1}(s)}{T_k(s)}-1\right)\Big/
\left(\frac{T_{k+1}(s)}{T_k(s)}+1\right).
\end{align}
By \cite[Eq.~(1.104)]{Riv90} we have
\begin{align} \label{Ratio-T-deriv}
\left(\frac{T_{k+1}(s)}{T_k(s)}\right)^\prime>0 \;\text{ for }\; s>1.
\end{align}
Since $T_{k+1}(1)/{T_k(1)}=1$ it follows from \eqref{Ratio-T-deriv} and \eqref{al-even} that $-\al_n$ strictly increases as $s$ increases.

Next, let $n=2k+1$ for some $k\in\bb N_0$. We introduce a new variable $\Theta$ by
\begin{align}
s=\cosh(\Theta).
\end{align}
Here $\Theta\in (0,\infty)$ for $s\in(1,\infty)$ and $s$ increases as $\Theta$ increases. Thus, it suffices to show that $-\al_n$ strictly increases as $\Theta$ increases.
Since (see \cite[Section~1.4]{Mas}) $T_n(s)= \cosh{(n\Theta)}$, \eqref{invGer}, \eqref{nr1} and \eqref{nr2} yield
\begin{align}
-\al_n &= \frac{1}{1+s}\left(1+\frac{T_{k+2}(s)}{T_{k+1}(s)}\right)-1 = \frac{1}{1+\cosh(\Theta)}\left(1+
\frac{\cosh((k+2)\Theta)}{\cosh((k+1)\Theta)}
\right)-1 \no
\\
&=\frac{\cosh((k+2)\Theta)-\cosh(\Theta)\cosh((k+1)\Theta)}
{(1+\cosh(\Theta))\cosh((k+1)\Theta)} \no
=\frac{\sinh{(\Theta)}}{1+\cosh(\Theta)}
\frac{\sinh{((k+1)\Theta)}}{\cosh((k+1)\Theta)}
\\
&=\tanh{(\Theta/2)}\tanh((k+1)\Theta). \label{tanh}
\end{align}
The derivative of the expression in \eqref{tanh} is
\begin{align}\label{tanh2}
\frac{\mathrm{sech}^2(\Theta/2)\tanh((k+1)\Theta)}{2} + (k+1)\tanh(\Theta/2)\mathrm{sech}^2((k+1)\Theta)>0
\end{align}
hence, by \eqref{tanh} and \eqref{tanh2}, $-\al_n$ strictly increases as $\Theta$ increases.\\

(iii) Let $0<\gamma_1<\gamma_2<\pi$ and denote by $\Phi_k$ and $P_k$ the $k$-th monic orthogonal polynomials for $\mu_{K_\ga}$ and $\mu_{L_\ga}$, respectively.
By \eqref{W-def}, the required inequality \eqref{Arc-W-ga-mon} is equivalent to
\begin{align}\label{widgam2}
\frac{[W_{2,n}(\mu_{K_{\gamma_2}})]^2}{[W_{2,n}(\mu_{K_{\gamma_1}})]^2} = \frac{\|\Phi_n\|_{L_2(\mu_{K_{\ga_2}})}^2\ca(K_{\gamma_2})^{-2n}}
{\|\Phi_n\|_{L_2(\mu_{K_{\ga_1}})}^2\ca(K_{\gamma_1})^{-2n}}<1.
\end{align}
We will verify this inequality for $n$ even and odd separately. As a preliminary note that by Theorem~\ref{Thm:LB-R-sat},
\begin{align}\label{LB-impro}
\|P_k\|_{L_2(\mu_{L_\ga})}^2 = 2\ca(L_\ga)^{2k} = 2[\sin(\gamma/2)]^{4k}.
\end{align}

First, assume that $n=2k$ for some $k\in\bb N$. By \cite[Eq.~(13.1.15)]{Sim05} we have
\begin{align} \label{PkPhi2k}
\|P_k\|_{L_2(\mu_{L_\ga})}^2=2(1-\al_{2k-1})^{-1}\|\Phi_{2k}\|_{L_2(\mu_{K_\ga})}^2.
\end{align}
Combining \eqref{PkPhi2k} with \eqref{LB-impro} and \eqref{Arc-cap} for $\gamma=\gamma_1$ and $\gamma=\gamma_2$ shows that \eqref{widgam2} is equivalent to
\begin{align}\label{widgam3}
\frac{1-\al_{2k-1}(\mu_{K_{\gamma_2}}) }{1-\al_{2k-1}(\mu_{K_{\gamma_1}}) }<1.
\end{align}
The inequality \eqref{widgam3} holds true by part (ii) and the fact that the $\al_j$'s are negative.

Next, let $n=2k+1$ for some $k\in\bb N_0$. By \cite[Eq.~(13.1.21)]{Sim05} we have
\begin{align} \label{PkPhi2k-1}
\|P_{k+1}\|_{L_2(\mu_{L_\ga})}^2=2(1+\al_{2k+1})\|\Phi_{2k+1}\|_{L_2(\mu_{K_\ga})}^2.
\end{align}
Combining \eqref{PkPhi2k-1} with \eqref{LB-impro} and \eqref{Arc-cap} for $\gamma=\gamma_1$ and $\gamma=\gamma_2$ shows that \eqref{widgam2} is equivalent to
\begin{align}\label{final}
\left(\frac{\sin^2(\ga_2/2)}{1+\al_{2k+1}(\mu_{K_{\gamma_2}})}\right)\Big/
\left(\frac{\sin^2{(\ga_1/2)} }{1+\al_{2k+1}(\mu_{K_{\gamma_1}})}\right)<1.
\end{align}
To prove \eqref{final} it suffices to show that
\begin{align}
\frac{\sin^2{(\ga/2)} }{1+\al_{2k+1}(\mu_{K_{\gamma}})}
\end{align}
strictly decreases as $\gamma$ increases. This is equivalent to showing that
\begin{align}\label{final 3}
\frac{1+\al_{2k+1}(\mu_{K_{\gamma}})}{\sin^2{(\ga/2)} }
\end{align}
is strictly decreasing when $s$, defined in \eqref{s-c-def}, is increasing. Note that
\begin{align}\label{final 4}
\sin^2{(\ga/2)}=a=2/(1+s).
\end{align}
It follows from \eqref{invGer}, \eqref{nr1}, \eqref{nr2} that
\begin{align}\label{final 5}
1+\al_{2k+1}(\mu_{K_{\gamma}}) = \frac{T_{k+1}(s)+2sT_{k+1}(s)-T_{k+2}(s)}{(1+s)T_{k+1}(s)}.
\end{align}
Combining \eqref{final 5} and \eqref{final 4} we get
\begin{align}
\frac{1+\al_{2k+1}(\mu_{K_{\gamma}})}{\sin^2{(\ga/2)} }=\frac{1}{2}+\frac{2sT_{k+1}(s)-T_{k+2}(s)}{2 T_{k+1}(s)}.
\end{align}
The three-term recurrence relation $2sT_{k+1}(s)-T_{k+2}(s)=T_{k}(s)$ then implies
\begin{align}\label{final 7}
\frac{1+\al_{2k+1}(\mu_{K_{\gamma}})}{\sin^2{(\ga/2)} }=\frac{1}{2}+\frac{T_k(s)}{2 T_{k+1}(s)}.
\end{align}
It follows from \eqref{final 7} and \eqref{Ratio-T-deriv} that
the expression in \eqref{final 3} strictly decreases as $s$ increases. This completes the proof.
\end{proof}



\end{document}